\newif\ifpictures
\picturestrue

\newif\ifcomment
\commenttrue

\documentclass[12pt]{amsart}
\usepackage{latex_base}
\usepackage{macros}

\author{Khazhgali Kozhasov}

\address{Khazhgali Kozhasov, Technische Universit\"at Braunschweig, Institut f\"ur Analysis und Algebra, Universit\"atsplatz 2, 38106 Braunschweig,
 Germany\medskip}

\email{k.kozhasov@tu-braunschweig.de}



\title[On eigenvalues of locally PSD matrices]{On eigenvalues of symmetric matrices with PSD principal submatrices
}

\begin{document}

\maketitle

\begin{changemargin}{1.5cm}{1.5cm}
  {\bf\noindent Abstract.}
  We investigate convexity properties of the set of eigenvalue tuples of $n\times n$ real symmetric matrices, whose all $k\times k$ (where $k\leq n$ is fixed) minors are positive semidefinite.
  It is proven that the set $\lambda(\cS^{n,k})$ of eigenvalue vectors of all such matrices is star-shaped with respect to the nonnegative orthant $\R^n_{\geq 0}$ and not convex already when $(n,k)=(4,2)$.
\end{changemargin}
\vspace{1cm}

\section{Introduction and statement of results}

Let $\Sym_n$ denote the space of $n\times n$ real symmetric matrices. A matrix $X\in \Sym_n$ is called \emph{positive semidefinite} (or, simply, PSD), if
$\Vector{v}^T X\Vector{v} \geq 0$
holds for all column vectors $\Vector{v}\in \R^n$.
Equivalently, $X\in \Sym_n$ is PSD if all its eigenvalues are nonnegative.
For $k\in \{1,\dots, n\}$ let $\cS^{n,k}$ denote the set of those matrices in $\Sym_n$, whose all $k\times k$ principal submatrices are positive semidefinite.
Elements of $\cS^{n,k}$ are known as \emph{$k$-locally positive semidefinite matrices}, they have been systematically studied in \cite{BDMS2020} and \cite{BDSS2020}.
Note that $\cS^{n,n}$ consists of all positive semidefinite matrices in $\Sym_n$.
The set $\cS^{n,k}$ is a closed convex cone containing $\cS^{n,n}$ and, moreover, $\cS^{n,k}\supsetneq \cS^{n,k+1}$ for any $k\in \{1,\dots, n-1\}$.
The dual cone to $\cS^{n,k}\subset \Sym_n$ consists of PSD matrices of factor width $k$, which were introduced and investigated in \cite{BOMAN2005239, PP2018, GKS2019}. Matrices of factor width $2$ are also known as scaled diagonally dominant matrices, see \cite{BOMAN2005239}.
One of the motivations to study properties of locally PSD matrices came from empirical observations that some optimization problems (see \cite{SL2014, KDS2016}), where constraints are given by PSD matrices, have close optimal values to those of their relaxations, where constraints are relaxed to be $k$-locally PSD. 

Given a matrix $X\in \Sym_n$, let $\lambda(X)=(\lambda_1(X),\dots,\lambda_n(X))$,  $\lambda_1(X)\geq\dots\geq \lambda_n(X)$, denote the vector of its ordered eigenvalues and for any permutation on $n$ elements $\sigma \in S_n$ let $\lambda(X)^\sigma=(\lambda_{\sigma_1}(X),\dots, \lambda_{\sigma_n}(X))$ be the vector of permuted eigenvalues.
In this work we are interested in the set
\begin{align*}
  \lambda(\cS^{n,k})\ =\ \left\{\lambda(X)^\sigma\,:\, X\in \cS^{n,k},\ \sigma\in S_n\right\}
\end{align*}
of all possible (permuted) eigenvalue vectors of $k$-locally PSD matrices.
As observed in \cite{BDSS2020}, the set $\lambda(\cS^{n,k})$ is contained in \emph{the closed hyperbolicity cone}
\begin{align*}
  H(e^n_k)\ =\ \{\lambda\in \R^n\,:\, e^n_k(\lambda+t(1,\dots,1))=0\,\Rightarrow t\leq 0\}
\end{align*}
of the $k$-th elementary symmetric polynomial $e^n_k(\lambda)=\sum_{1\leq i_1<\dots<i_k\leq n} \lambda_{i_1}\cdot\ldots\cdot \lambda_{i_k}$.
Moreover, the inclusion $\lambda(\cS^{n,k})\subseteq H(e^n_k)$ is strict for $2<k<n-1$, for $(n,k)=(4,2)$ \cite[Cor. 3]{BDSS2020} and, surprisingly, one has equality $\lambda(\cS^{n,n-1})=H(e^n_{n-1})$ for $k=n-1$ (see \cite[Thm. $2.2$]{BDSS2020}).
The authors of \cite{BDSS2020} pointed out that it is unknown whether the set $\lambda(\cS^{n,k})$ of eigenvalue vectors of matrices in $\cS^{n,k}$ is convex for any $k$ and $n$.
We give a negative answer to this question, showing that $\lambda(\cS^{4,2})\subsetneq H(e^4_2)$ is not convex.
\begin{theorem}\label{thm:nonconvex}
  The set
  \begin{align*}
    \lambda(\cS^{4,2})\ =\ \{(\lambda_{\sigma_1}(X),\lambda_{\sigma_2}(X),\lambda_{\sigma_3}(X),\lambda_{\sigma_4}(X))\,:\, X\in\cS^{4,2},\ \sigma\in S_4\} 
  \end{align*}
  of eigenvalue vectors of $2$-locally PSD matrices in $\Sym_4$ is not convex.
\end{theorem}

\begin{remarkx}\label{rem:(4,2)}
  Specifically, we prove that  $(4,4,-1,-1)$ is not in $\lambda(\cS^{4,2})$, though it lies on the segment joining the vector $(4,4,0,-2)\in \lambda(\cS^{4,2})$ with its permuted copy.
\end{remarkx}

Even though the set $\lambda(\cS^{n,k})$ is not convex in general, it is star-shaped with respect to the nonnegative orthant $\R^n_{\geq 0}=\lambda(\cS^{n,n})$.

\begin{proposition}\label{prop:star}
For all $\lambda\in \lambda(\cS^{n,k})$, $\lambda^+\in \R^n_{\geq 0}$ the sum $\lambda+\lambda^+$ is in $\lambda(\cS^{n,k})$.
\end{proposition}

We now discuss one application of this result.
It is shown in \cite{BDSS2020} that every point on the boundary of $H(e^4_2)$ with exactly one negative entry is an eigenvalue vector of some matrix in $\cS^{4,2}$.
We observe that the same holds for all points in $H(e^4_2)$ with this property.

\begin{corollary}\label{cor:(4,2)}
Any $\lambda \in H(e^4_2)$ with at most $1$ negative entry is an eigenvalue vector of some matrix in $\cS^{4,2}$.
\end{corollary}

In the rather trivial case $k=1$ the set $\lambda(\cS^{n,1})=H(e^n_1)$ is a closed half-space.

\begin{theorem}\label{thm:k=1}
  For any $n\geq 1$ we have
  \begin{align*}
    \lambda(\cS^{n,1})\ =\ \left\{\lambda\in \R^n\,:\,e^n_1(\lambda)=\sum_{j=1}^n \lambda_j\geq 0\right\}
  \end{align*}
\end{theorem}




\section{Some open questions}

By its definition, the set $\lambda(\cS^{n,k})$ consists of permuted eigenvalue vectors of $k$-locally PSD matrices.
It is also natural to consider the subset
\begin{align*}
  \lambda_{\leq}(\cS^{n,k})\ =\ \{\lambda(X)\,:\, X\in \cS^{n,k}\}
\end{align*}
of ordered eigenvalue vectors, that is, the intersection of $\lambda(\cS^{n,k})$ with half-spaces cut out by inequalities $\lambda_i\geq \lambda_{i+1}$, $i=1,\dots, n-1$.  
Our proof of Theorem \ref{thm:nonconvex} does not imply that $\lambda_{\leq}(\cS^{4,2})$ is non-convex, see Remark \ref{rem:(4,2)}.
It would be interesting to understand whether $\lambda_{\leq}(\cS^{n,k})$ is non-convex for some $(n,k)$.
When $(n,k)=(n,n-1)$ the set $\lambda_\leq(\cS^{n,k}))$ is a convex cone by \cite[Thm. $2.2$]{BDSS2020} and when $(n,k)=(n,1)$ it is the convex polyhedral cone
  \begin{align*}
    \lambda_{\leq}(\cS^{n,1})\ =\ \left\{\lambda\in \R^n\,:\,\lambda_1\geq \dots\geq \lambda_n,\ \sum_{j=1}^i\lambda_j\geq 0\ \textrm{for all}\ i=1,\dots,n\right\},
  \end{align*}
  see the proof of Theorem \ref{thm:k=1} below.
 
  Nonzero vectors in the hyperbolicity cone $H(e^n_k)$ can have at most $n-k$ negative entries. Generalizing \cite[Lemma $2$]{BDSS2020} one can show that if a vector $\lambda\in H(e^n_k)$ has $n-k$ negative entries, then the other $k$ entries must be positive. 
  One interesting question is to understand whether for any $(n,k)$ there exist $k$-locally PSD matrices with $n-k$ negative eigenvalues.
  In the cases $(n,n-1)$ and $(n,1)$ \cite[Thm. $2.2$]{BDSS2020} and Theorem \ref{thm:k=1} imply a positive answer to this question. Also, one can see that the real symmetric matrix
  \begin{align*}
    X\ =\
    \begin{pmatrix}
      346 & 84  & -16 & -98\\
      84  & 240 & 77  & 156\\
      -16 & 77  & 30  &  70\\
      -98 & 156 & 70  & 170
    \end{pmatrix}
  \end{align*}
  is in $\cS^{4,2}$ and has two negative eigenvalues.

  See also Remark \ref{rem:sos} for an interesting polynomial optimization problem.
  
\section{Acknowledgements}
We would like to thank Grigoriy Blekherman, whose talk at the POEMA 3rd Workshop inspired this work. We wish to thank Boulos El Hilany for his helpful insight with the SageMath computation.

\section{Proofs}

We need the following elementary fact that we prove first.

\begin{lemma}\label{lemma1}
If $(\lambda_1,\dots, \lambda_{n'})$ is in $\lambda(\cS^{n',k})$, then $(\lambda_1,\dots,\lambda_{n'},\underbrace{0,\dots, 0}_{n-n'})$ is in $\lambda(\cS^{n,k})$. 
\end{lemma}

\begin{proof}
  Let $(\lambda_1,\dots,\lambda_{n'})=\lambda(X')$, where $X'\in \cS^{n',k}$.
  Note that the $n\times n$ matrix
  \begin{align*}
    X\ =\
    \begin{pmatrix}
      X' & {\Huge 0} \\
      {\Huge 0} & {\Huge 0} 
    \end{pmatrix}
  \end{align*}
  is $k$-locally PSD, since its principal submatrices $X|_S$, $S\subseteq \{1,\dots, n\}$, $\vert S\vert = k$, read
  \begin{align*}
      X|_S\ =\
      \begin{pmatrix}
        X'|_{S'} & {\Huge 0}\\
        {\Huge 0} & {\Huge 0}
      \end{pmatrix},
  \end{align*}
  where $S'=S\cap \{1,\dots, n'\}$.
  It is easy to see that $(\lambda_1,\dots,\lambda_{n'}, 0,\dots, 0)=\lambda(X)$.
\end{proof}

\begin{proofof}{Theorem \ref{thm:nonconvex}}
  Let us observe that the vector $\lambda=(4,4,-1,-1)$ belongs to the cone $H(e^4_2)$. Indeed, for any $t\geq 0$ the value
  \begin{align*}
    e^4_2(4+t,4+t,-1+t,-1+t)\ &=\ (4+t)^2+(t-1)^2+4(4+t)(t-1)\\
    &=\ 1+18t+6t^2
  \end{align*}
  is positive and hence the vector $\lambda+t(1,1,1,1)$ lies in $H(e^4_2)$.
  Furthermore, $\lambda$ is the mid-point of the segment joining $\lambda'=(4,4,-2,0)$ and $\lambda''=(4,4,0,-2)$. Note that $\lambda'$ is obtained by appending a zero to the vector $(4,4,-2)\in H(e^3_2)=\lambda(\cS^{3,2})$, where equality of sets is due to \cite[Thm. $2.2$]{BDSS2020}. Thus, Lemma \ref{lemma1} implies that $\lambda'$ (and hence also its permuted copy $\lambda''$) belongs to $\lambda(\cS^{4,2})$. 
  We show that $\lambda$ is not in $\lambda(\cS^{4,2})$.
  Note that $\lambda = \lambda(X)$ for some matrix $X\in \Sym_4$ if and only if we can write
  \begin{equation}\label{eq:rep}
  \begin{aligned}
    X\ =\ V^T \Lambda V\ &=\
    \begin{pmatrix}
      \Vector{v}^T_1 \,|\, \Vector{v}^T_2\, |\,  \Vector{v}^T_3\, |\, \Vector{v}^T_4 
    \end{pmatrix}
    \begin{pmatrix}
      4 &   &   & {\huge 0} \\
      & 4 &   & \\
      &   & -1 & \\
     {\huge 0} &   &    & -1
    \end{pmatrix}
                 \begin{pmatrix}
                   \Vector{v}_1 \\ \Vector{v}_2 \\ \Vector{v}_3\\ \Vector{v}_4
                 \end{pmatrix}\\
    &=\
      \begin{pmatrix}
  \substack{4 v_{11}^2 + 4 v_{21}^2\\ - v_{31}^2 - v_{41}^2} & 
  \substack{4 v_{11} v_{12} + 4 v_{21} v_{22}\\ - v_{31} v_{32} - v_{41} v_{42}} & 
  \substack{4 v_{11} v_{13} + 4 v_{21} v_{23}\\ - v_{31} v_{33} - v_{41} v_{43}} & 
  \substack{4 v_{11} v_{14} + 4 v_{21} v_{24}\\ - v_{31} v_{34} - v_{41} v_{44}} \\
  \substack{4 v_{11} v_{12} + 4 v_{21} v_{22}\\ - v_{31} v_{32} - v_{41} v_{42}} &
  \substack{4 v_{12}^2 + 4 v_{22}^2\\ - v_{32}^2 - v_{42}^2} &
  \substack{4 v_{12} v_{13} + 4 v_{22} v_{23}\\ - v_{32} v_{33} - v_{42} v_{43}} &
  \substack{4 v_{12} v_{14} + 4 v_{22} v_{24}\\ - v_{32} v_{34} - v_{42} v_{44}} \\
  \substack{4 v_{11} v_{13} + 4 v_{21} v_{23}\\ - v_{31} v_{33} - v_{41} v_{43}} &
  \substack{4 v_{12} v_{13} + 4 v_{22} v_{23}\\ - v_{32} v_{33} - v_{42} v_{43}} &
  \substack{4 v_{13}^2 + 4 v_{23}^2\\ - v_{33}^2 - v_{43}^2} &
  \substack{4 v_{13} v_{14} + 4 v_{23} v_{24}\\ - v_{33} v_{34} - v_{43} v_{44}}\\
  \substack{4 v_{11} v_{14} + 4 v_{21} v_{24}\\ - v_{31} v_{34} - v_{41} v_{44}} &
  \substack{4 v_{12} v_{14} + 4 v_{22} v_{24}\\ - v_{32} v_{34} - v_{42} v_{44}} &
  \substack{4 v_{13} v_{14} + 4 v_{23} v_{24}\\ - v_{33} v_{34} - v_{43} v_{44}} &
  \substack{4 v_{14}^2 + 4 v_{24}^2\\ - v_{34}^2 - v_{44}^2}
      \end{pmatrix}
  \end{aligned},
\end{equation}
where $V\in O(4)$ is an orthogonal matrix, whose rows are eigenvectors $\Vector{v}_1,\Vector{v}_2, \Vector{v}_3, \Vector{v}_4\in \R^4$ of $X$.
The representation \eqref{eq:rep} is unique up to the choice of an orthonormal basis $\{\Vector{v}_1, \Vector{v}_2\}$ (respectively, $\{\Vector{v}_3, \Vector{v}_4\}$) of the eigenspace of $X$ corresponding to the eigenvalue $4$ (respectively, $-1$).
In particular, performing orthogonal changes of bases in $L=\mathrm{\normalfont{Span}}\{\Vector{v}_1,\Vector{v}_2\}$ and in $L^\perp=\mathrm{\normalfont{Span}}\{\Vector{v}_3,\Vector{v}_4\}$  we can put $v_{12}=v_{34}=0$.
  Now, the matrix \eqref{eq:rep} with $v_{12}=v_{34}=0$ is $2$-locally PSD if and only if the polynomial system 

  {\footnotesize\begin{equation*}
  \begin{aligned}
4 v_{11}^2 + 4 v_{21}^2 - v_{31}^2 - v_{41}^2-m_1\ =\ 0,\\
 4 v_{22}^2 - v_{32}^2 - v_{42}^2-m_2\ =\ 0,\\
4 v_{13}^2 + 4 v_{23}^2 - v_{33}^2 - v_{43}^2-m_3\ =\ 0,\\
4 v_{14}^2 + 4 v_{24}^2 - v_{44}^2-m_4\ =\ 0,\\ 
(4 v_{11}^2 + 4 v_{21}^2 - v_{31}^2 - v_{41}^2) (4 v_{22}^2 - v_{32}^2 - v_{42}^2)-(4 v_{21} v_{22} - v_{31} v_{32} - v_{41} v_{42})^2-m_{12}\ =\  0,\\
(4 v_{11}^2 + 4 v_{21}^2 - v_{31}^2 - v_{41}^2) (4 v_{13}^2 + 4 v_{23}^2 - v_{33}^2 - v_{43}^2)-(4 v_{11} v_{13} + 4 v_{21} v_{23} - v_{31} v_{33} - v_{41} v_{43})^2-m_{13}\ =\ 0,\\
(4 v_{11}^2 + 4 v_{21}^2 - v_{31}^2 -v_{41}^2) (4 v_{14}^2 + 4 v_{24}^2 - v_{44}^2)-(4 v_{11} v_{14} + 4 v_{21} v_{24} - v_{41} v_{44})^2-m_{14}\ =\ 0,\\
(4 v_{22}^2 - v_{32}^2 - v_{42}^2) (4 v_{13}^2 + 4 v_{23}^2 - v_{33}^2 - v_{43}^2)-(4 v_{22} v_{23} - v_{32} v_{33} - v_{42} v_{43})^2-m_{23}\ =\ 0,\\
(4 v_{22}^2 - v_{32}^2 - v_{42}^2) (4 v_{14}^2 + 4 v_{24}^2 - v_{44}^2)-(4 v_{22} v_{24} - v_{42} v_{44})^2-m_{24}\ =\ 0,\\
(4 v_{13}^2 + 4 v_{23}^2 - v_{33}^2 - v_{43}^2) (4 v_{14}^2 + 4 v_{24}^2 - v_{44}^2)-(4 v_{13} v_{14} + 4 v_{23} v_{24} - v_{43} v_{44})^2-m_{34}\ =\ 0,\\
 v_{11}^2 + v_{21}^2 + v_{31}^2 + v_{41}^2 - 1 = 0,\\
 v_{22}^2 + v_{32}^2 + v_{42}^2 - 1 = 0,\\ 
  v_{13}^2 + v_{23}^2 + v_{33}^2 + v_{43}^2 - 1 = 0,\\
 v_{14}^2 + v_{24}^2 + v_{44}^2 - 1 = 0,\\ 
  v_{21} v_{22} + v_{31} v_{32} + v_{41} v_{42} = 0,\\ 
 v_{11} v_{13} +  v_{21} v_{23} + v_{31} v_{33} + v_{41} v_{43} = 0,\\ 
 v_{11} v_{14} + v_{21} v_{24} + v_{41} v_{44} = 0,\\ 
 v_{22} v_{23} + v_{32} v_{33} + v_{42} v_{43} = 0,\\
 v_{22} v_{24} + v_{42} v_{44} = 0,\\
 v_{13} v_{14} + v_{23} v_{24} + v_{43} v_{44} = 0,
  \end{aligned}
\end{equation*}}

\noindent
has a real solution with nonnegative $m_1, m_2, m_3, m_4, m_{12}, m_{13}, m_{14}, m_{23}, m_{24}, m_{34}$. These variables represent principal minors of $X$ of sizes $1\times 1$ and $2\times 2$, and the last $10$ equations encode orthogonality of the matrix $V$.
Let $I\vartriangleleft\Q[v_{ij}, m_i, m_{ij}]$ be the ideal generated by the $20$ polynomials of the system. Using $\mathrm{\normalfont{SageMath}}$ \cite{sagemath} we find out that the elimination ideal $J=I\cap \Q[m_i,m_{ij}]$ of $I$ with respect to the variables $v_{ij}$ is generated by polynomials
{\footnotesize\begin{equation*}
    \begin{aligned}
L\ =\ &m_{12} + m_{13} + m_{14} + m_{23} + m_{24} + m_{34} - 1,\\
   f_1\ =\ &3m_1 + m_{23} + m_{24} + m_{34} - 5,\\
   f_2\ =\ &3m_2 + m_{13} + m_{14} + m_{34} - 5,\\
   f_3\ =\ &3m_3 - m_{13} - m_{23} - m_{34} - 4,\\
   f_4\ =\ &3m_4 - m_{14} - m_{24} - m_{34} - 4,\\
   F\ =\ &- 2816 - 256m_{13} - 256m_{14} - 256m_{23} - 256m_{24} - 224m_{34} - 50m_{34}^3\\
   &+ 192m_{14}m_{24} + 256m_{13}^2 + 192m_{13}m_{14} + 256m_{14}^2 + 192m_{13}m_{23} + 160m_{14}m_{23} + 256m_{23}^2 \\
&  +160m_{13}m_{24}+192m_{23}m_{24} + 256m_{24}^2 + 96m_{13}m_{34} +
96m_{14}m_{34} + 96m_{23}m_{34} + 96m_{24}m_{34} + 249m_{34}^2\\
&+ 64m_{13}^2m_{14} + 64m_{13}m_{14}^2 + 64m_{13}^2m_{23} + 112m_{13}m_{14}m_{23} + 128m_{14}^2m_{23} + 64m_{13}m_{23}^2\\
&+ 128m_{14}m_{23}^2 +
128m_{13}^2m_{24}+ 112m_{13}m_{14}m_{24} + 64m_{14}^2m_{24} + 112m_{13}m_{23}m_{24} +112m_{14}m_{23}m_{24}\\
&+ 64m_{23}^2m_{24} + 128m_{13}m_{24}^2 + 64m_{14}m_{24}^2 +
64m_{23}m_{24}^2 + 128m_{13}^2m_{34} + 240m_{13}m_{14}m_{34} + 128m_{14}^2m_{34} \\
& +240m_{13}m_{23}m_{34} + 206m_{14}m_{23}m_{34} + 128m_{23}^2m_{34} + 206m_{13}m_{24}m_{34} +
240m_{14}m_{24}m_{34} \\
&+ 240m_{23}m_{24}m_{34} + 128m_{24}^2m_{34} + 78m_{13}m_{34}^2 +78m_{14}m_{34}^2 + 78m_{23}m_{34}^2 + 78m_{24}m_{34}^2\\
& + 16m_{13}^2m_{14}m_{23} + 16m_{13}m_{14}^2m_{23} + 16m_{13}m_{14}m_{23}^2 +
25m_{14}^2m_{23}^2 + 16m_{13}^2m_{14}m_{24} + 16m_{13}m_{14}^2m_{24}\\
&+ 16m_{13}^2m_{23}m_{24}  + 14m_{13}m_{14}m_{23}m_{24} + 16m_{14}^2m_{23}m_{24} + 16m_{13}m_{23}^2m_{24} +
16m_{14}m_{23}^2m_{24} + 25m_{13}^2m_{24}^2\\
&+ 16m_{13}m_{14}m_{24}^2 + 16m_{13}m_{23}m_{24}^2 + 16m_{14}m_{23}m_{24}^2 + 16m_{13}^2m_{14}m_{34} + 16m_{13}m_{14}^2m_{34} \\
&+ 16m_{13}^2m_{23}m_{34} + 82m_{13}m_{14}m_{23}m_{34} + 50m_{14}^2m_{23}m_{34} +16m_{13}m_{23}^2m_{34} + 50m_{14}m_{23}^2m_{34}\\
& + 50m_{13}^2m_{24}m_{34} + 82m_{13}m_{14}m_{24}m_{34} + 16m_{14}^2m_{24}m_{34} + 82m_{13}m_{23}m_{24}m_{34} +82m_{14}m_{23}m_{24}m_{34} \\
&+ 16m_{23}^2m_{24}m_{34} + 50m_{13}m_{24}^2m_{34} +
16m_{14}m_{24}^2m_{34} + 16m_{23}m_{24}^2m_{34} + 25m_{13}^2m_{34}^2 + 66m_{13}m_{14}m_{34}^2\\
& +25m_{14}^2m_{34}^2 + 66m_{13}m_{23}m_{34}^2 + 100m_{14}m_{23}m_{34}^2 + 25m_{23}^2m_{34}^2
+ 100m_{13}m_{24}m_{34}^2 + 66m_{14}m_{24}m_{34}^2\\
&+ 66m_{23}m_{24}m_{34}^2 +25m_{24}^2m_{34}^2 + 50m_{13}m_{34}^3 + 50m_{14}m_{34}^3 + 50m_{23}m_{34}^3 +
50m_{24}m_{34}^3 + 25m_{34}^4. 
  \end{aligned}
\end{equation*}}
Note that $L$ and $F$ depend only on the variables $m_{ij}$.
In particular, if $(m_i, m_{ij})$ is a nonnegative zero of $J$, then $(m_{ij})$ is a nonnegative zero of $L=F=0$. 
Vice versa, if $L(m_{ij})=F(m_{ij})=0$ for some nonnegative $m_{ij}$, setting
\begin{equation*}
  \begin{aligned}
    m_1\ &=\ \frac{1}{3}(m_{12}+m_{13}+m_{14}+4)\geq 0,\\
    m_2\ &=\ \frac{1}{3}(m_{12}+m_{23}+m_{24}+4)\geq 0,\\
    m_3\ &=\ \frac{1}{3}(m_{13}+m_{23}+m_{34}+4)\geq 0,\\
    m_4\ &=\ \frac{1}{3}(m_{14}+m_{24}+m_{34}+4)\geq 0,
  \end{aligned}
\end{equation*}
we recover a nonnegative zero of $J$.

Finally, we prove that $L$ and $F$ have no common nonnegative zeros.
For this it is enough to show that $F$ (which does not depend on $m_{12}$) has no zeros in the unit  simplex
\begin{align*}
\Delta\ =\ \{(m_{13}, m_{14}, m_{23}, m_{24}, m_{34})\in \R^5_{\geq 0}\,:\, m_{13}+m_{14}+m_{23}+m_{24}+m_{34}\leq 1\}.
\end{align*}

Since $0\leq m_{ij}^4\leq m^3_{ij}\leq m_{ij}^2\leq m_{ij}^1\leq 1$ hold for points in $\Delta$, term-wise estimates give

{\footnotesize\begin{equation*}
\begin{aligned}
  F\ =\ &-2816 -256(m_{13}+m_{14}+m_{23}+m_{24}+m_{34})-50m_{34}^3-7m_{34}^2\\
  &+32m_{34} +192(m_{14}m_{24}+m_{13}m_{14}+m_{13}m_{23}+m_{23}m_{24})+256(m_{13}^2+m_{14}^2+m_{23}^2+m_{24}^2+m_{34}^2)\\
  &+160(m_{14}m_{23}+m_{13}m_{24}) + 96(m_{13}m_{34}+m_{14}m_{34}+m_{23}m_{34}+m_{24}m_{34})\\
  &+ 64(m_{13}^2m_{14} + m_{13}m_{14}^2 + m_{13}^2m_{23} + m_{13}m_{23}^2 + m_{14}^2m_{24} + m_{23}^2m_{24}+ m_{14}m_{24}^2 + m_{23}m_{24}^2)\\
  &+ 112(m_{13}m_{14}m_{23}+ m_{13}m_{14}m_{24}+ m_{13}m_{23}m_{24} + m_{14}m_{23}m_{24}) \\
  &+ 128(m_{14}^2m_{23} + m_{14}m_{23}^2 + m_{13}^2m_{24} + m_{13}m_{24}^2 + m_{13}^2m_{34}+ m_{14}^2m_{34}+ m_{23}^2m_{34}+ m_{24}^2m_{34})\\
  &+ 240(m_{13}m_{14}m_{34} + m_{13}m_{23}m_{34}+m_{14}m_{24}m_{34} + m_{23}m_{24}m_{34})\\
  &+ 206 (m_{14}m_{23}m_{34}  + m_{13}m_{24}m_{34})  + 78(m_{13}m_{34}^2 + m_{14}m_{34}^2 + m_{23}m_{34}^2 + m_{24}m_{34}^2)\\
  & + 16(m_{13}^2m_{14}m_{23} + m_{13}m_{14}^2m_{23} + m_{13}m_{14}m_{23}^2 + m_{13}^2m_{14}m_{24}+ m_{13}m_{14}^2m_{24}+ m_{13}^2m_{23}m_{24}\\
  &+ m_{14}^2m_{23}m_{24} + m_{13}m_{23}^2m_{24} + m_{14}m_{23}^2m_{24}+ m_{13}m_{14}m_{24}^2 + m_{13}m_{23}m_{24}^2 + m_{14}m_{23}m_{24}^2 + m_{13}^2m_{14}m_{34}\\
  &+ m_{13}m_{14}^2m_{34} + m_{13}^2m_{23}m_{34} + m_{13}m_{23}^2m_{34}+ m_{14}^2m_{24}m_{34}+ m_{23}^2m_{24}m_{34} +
  m_{14}m_{24}^2m_{34} + m_{23}m_{24}^2m_{34})\\
  &+ 25(m_{14}^2m_{23}^2+ m_{13}^2m_{24}^2+ m_{13}^2m_{34}^2+ m_{14}^2m_{34}^2+ m_{23}^2m_{34}^2 + m_{24}^2m_{34}^2+ m_{34}^4)\\
  &+ 82(m_{13}m_{14}m_{23}m_{34}+ m_{13}m_{14}m_{24}m_{34}+ m_{13}m_{23}m_{24}m_{34} + m_{14}m_{23}m_{24}m_{34})+ 14m_{13}m_{14}m_{23}m_{24}\\
  &+ 50(m_{14}^2m_{23}m_{34}  + m_{14}m_{23}^2m_{34} + m_{13}^2m_{24}m_{34}  + m_{13}m_{24}^2m_{34}+ m_{13}m_{34}^3 + m_{14}m_{34}^3 + m_{23}m_{34}^3 + m_{24}m_{34}^3)\\
  &+ 66(m_{13}m_{14}m_{34}^2 + m_{13}m_{23}m_{34}^2+ m_{14}m_{24}m_{34}^2+ m_{23}m_{24}m_{34}^2) + 100(m_{14}m_{23}m_{34}^2 + m_{13}m_{24}m_{34}^2)\\   
  \leq\ &  -2816-256 m - 50m_{34}^3 -7m_{34}^2 + 32m_{34} + 192(m_{14}+m_{13}+m_{23}+m_{24}) + 256 m \\
  &+ 160 (m_{14}+m_{13})+96(m_{13}+m_{14}+m_{23}+m_{24}) + 64(m_{13}^2+m_{14}^2+m_{23}^2+m_{24}^2)\\
  &+ 112(m_{13}+m_{14}+m_{23}+m_{24})+ 128(m_{14}^2+ m_{23}^2+ m_{13}^2 + m_{24}^2) + 240(m_{14}+m_{13}+m_{24}+m_{23})\\
  &+ 206(m_{23}+m_{24}) + 78(m_{13}+m_{14}+m_{23}+m_{24}) + 16(m_{13}m_{14}m_{23}+m_{13}m_{14}m_{24}\\
  & +m_{13}m_{23}m_{24}+m_{14}m_{23}m_{24}+m_{13}m_{14}m_{34}+m_{13}m_{23}m_{34}+m_{14}m_{24}m_{34}+m_{23}m_{24}m_{34})\\
  & +25(m_{14}^2+m_{13}^2+m_{23}^2+m_{24}^2+m_{34}^2) + 82(m_{13} + m_{14} + m_{23} + m_{24})+14\\
  & + 50(m_{14}^2+m_{23}^2+m_{13}^2+m_{24}^2+m_{34}^2) + 66(m_{14}+m_{13}+m_{24}+m_{23})+100(m_{14}+m_{13})\\
  \leq\ & -2816 + 32 + 192 + 160 + 96 + 64 + 112 + 128 + 240 + 206 + 78 + 16 + 25 + 82 + 14 \\
  &+ 50 + 66 + 100\ =\ -1155\ <\ 0,
\end{aligned}
\end{equation*}}
where $m=m_{13}+m_{14}+m_{23}+m_{24}+m_{34}\leq 1$.
It follows that $F$ takes only negative values on $\Delta$.
The above implies that no matrix $X=V^T\Lambda V$ as in \eqref{eq:rep} can be $2$-locally PSD and hence $\lambda=(4,4,-1,-1)\notin \lambda(\cS^{4,2})$.
\end{proofof}

\begin{remarkx}\label{rem:sos}
Note that the absense of nonnegative roots of $L=F=0$ implies that the minimum $\ell_*$ of the linear function $\ell = m_{13}+m_{14}+m_{23}+m_{24}+m_{34}$ on $\{F=0\}\cap \R^5_{\geq 0}$ is greater than $1$. Numerical solving of a bunch of constrained optimization problems suggests that the minimum is $\ell_*=3.2$ and it is attained at $(m_{13}, m_{14}, m_{23}, m_{24}, m_{34})=(0,0,0,0,3.2)$. It would be interesting to find an algebraic certificate of nonnegativity (see e.g. \cite{Stengle} and \cite{DIW2017}) of say $\ell-2$ on the semialgebraic set $\{F=0\}\cap \R^5_{\geq 0}$. This would give an alternative proof of emptiness of $\{L=F=0\}\cap \R^5_{\geq 0}$ (cf. the proof of Theorem \ref{thm:nonconvex}).
\end{remarkx}

Proposition \ref{prop:star} essentially follows from properties of the PSD cone $\cS^{n,n}$.

\begin{proofof}{Proposition \ref{prop:star}}
  Let $X=V^T\Lambda V\in \cS^{n,k}$ be a spectral decomposition of a $k$-locally PSD matrix, where $V$ is an orthogonal matrix and $\Lambda$ is the diagonal matrix given by $\lambda\in \lambda(\cS^{n,k})$.
  Due to the orthogonal invariance of the PSD cone $\cS^{n,n}$, the matrix $X^+=V^T\Lambda^+V\in \cS^{n,n}\subset \cS^{n,k}$ is in particular $k$-locally PSD, where $\Lambda^+$ is the diagonal matrix given by the nonnegative vector $\lambda^+\in \R^n_{\geq 0}\subset \lambda(\cS^{n,k})$.
  By convexity of $\cS^{n,k}$  
  \begin{align*}
    X+ X^+\ =\ V^T(\Lambda+ \Lambda^+)V
  \end{align*}
is in $\cS^{n,k}$ and hence $\lambda+\lambda^+\in \lambda(\cS^{n,k})$.
\end{proofof}

\begin{proofof}{Corollary \ref{cor:(4,2)}}
  Let $\lambda=(\lambda_1,\lambda_2,\lambda_3,\lambda_4)\in H(e^4_2)$ satisfy $\lambda_1\geq \lambda_2\geq \lambda_3\geq \lambda_4$.
  If $\lambda_4\geq 0$, trivially $\lambda\in \R^4_{\geq 0}\subset \lambda(\cS^{4,2})$.
 Let then $\lambda_3\geq 0>\lambda_4$.
 Note that for $t=e^4_2(\lambda)/(\lambda_1+\lambda_2+\lambda_3)\geq 0$ the vector $\lambda-t(0,0,0,1)$ lies on the boundary of $H(e^4_2)$ and
 it has exactly one negative entry $\lambda_4-t$.
 By \cite[Thm. $8.1$]{BDSS2020} the vector $\lambda-t(0,0,0,1)$ is in $\lambda(\cS^{4,2})$.
 Thus, by Proposition \ref{prop:star} the original vector $\lambda=(\lambda-t(0,0,0,1))+t(0,0,0,1)$ is also in $\lambda(\cS^{4,2})$ because $t(0,0,0,1)\in \R^4_{\geq 0}$.
\end{proofof}

Our proof of Theorem \ref{thm:k=1} is based on the following classical result.

\begin{theorem}[Schur-Horn theorem]\label{Thm:SH}
  Let $x, \lambda\in \R^n$ be two vectors with $x_1\geq\dots\geq x_n$ and $\lambda_1\geq\dots\geq \lambda_n$.
  There exists a symmetric matrix $X\in \Sym_n$ with diagonal entries $x_1,\dots, x_n$ and with eigenvalues $\lambda_1,\dots,\lambda_n$ if and only if for all $i=1,\dots, n-1$
  \begin{align*}
    \sum_{j=1}^i x_j\ \leq\ \sum_{j=1}^i \lambda_j\quad\textrm{\normalfont{and}}\quad     \sum_{j=1}^n x_j\ =\ \sum_{j=1}^n \lambda_j.
  \end{align*}
\end{theorem}

\begin{proofof}{Theorem \ref{thm:k=1}}
Note that the statement is equivalent to 
    \begin{align}\label{eq:description}
    \lambda(\cS^{n,1})\ =\ \left\{\lambda^\sigma\in \R^n\,:\, \sigma\in S_n,\ \lambda_1\geq \dots\geq \lambda_n,\ \sum_{j=1}^i \lambda_j\geq 0\ \textrm{\normalfont{for all}}\ i=1,\dots, n\right\}.
  \end{align}
  Indeed, if $\sum_{j=1}^i \lambda_j<0$ for some $i<n$, then $\lambda_i\geq \lambda_{i+1}\geq \dots\geq \lambda_n$ must be negative. Therefore, $\sum_{j=1}^n\lambda_j\leq \sum_{j=1}^i\lambda_j<0$ and hence no permuted vector $\lambda^\sigma$, $\sigma\in S_n$, can satisfy $\sum_{j=1}^n\lambda_j^\sigma =\sum_{j=1}^n\lambda_j\geq 0$.
  The opposite direction is trivial. Thus, we need to prove that ordered vectors in $\lambda(\cS^{n,1})$ are exactly those that satisfy inequalities in \eqref{eq:description}.

  Let $\lambda=\lambda(X)$, where the matrix $X\in \cS^{n,1}$ has ordered diagonal entries $x_1\geq \dots\geq x_n$.
  By Theorem \ref{Thm:SH} we have that $\sum^i_{j=1} \lambda_j\geq \sum^i_{j=1} x_j\geq 0$ holds for $i=1,\dots, n$.

  Let now $\lambda\in \R^n$ be an ordered nonzero vector that satisfies inequalities from \eqref{eq:description}.
  There exists an index $r\in \{1,\dots, n\}$ so that $\lambda_1\geq \dots\geq \lambda_r>0\geq \lambda_{r+1}\geq \dots\geq \lambda_n$.
  Let us consider a vector $x\in \R^n$ defined by $x_1=\dots=x_r=\frac{1}{r}\sum^n_{j=1}\lambda_j\geq 0$ and $x_{r+1}=\dots=x_n=0$.
For $i=1,\dots, r-1$ we have
  \begin{align*}
    \sum_{j=1}^i x_j\ &=\ \frac{i}{r}\sum_{j=1}^n\lambda_j\ \leq\ \frac{i}{r}\sum_{j=1}^r\lambda_j\ =\ \sum_{j=1}^i \lambda_j+\frac{i-r}{r}\sum_{j=1}^i \lambda_j+\frac{i}{r}\sum_{j=i+1}^r \lambda_j\\
    &\leq\ \sum_{j=1}^i \lambda_j+\frac{i(i-r)}{r} \lambda_i +\frac{i}{r}\sum_{j=i+1}^r\lambda_j\ \leq\ \sum_{j=1}^i\lambda_j + \frac{i}{r}\sum_{j=i+1}^r (\lambda_j-\lambda_i)\ \leq \sum_{j=1}^i\lambda_j,
  \end{align*}
  where ordering $\lambda_1\geq \dots\geq \lambda_r>0\geq \lambda_{r+1}\geq \dots\geq \lambda_n$ is used in several places. For $i=r,\dots, n$ we have
  \begin{align*}
    \sum_{j=1}^i x_j\ =\ \sum^n_{j=1} \lambda_j\ \leq\ \sum_{j=1}^i \lambda_j,
  \end{align*}
  where we use negativity of $\lambda_{r+1},\dots, \lambda_n<0$.
  These inequalities, the equality $\sum_{j=1}^n x_j=\sum_{j=1}^n\lambda_j$ and Theorem \ref{Thm:SH} imply that $\lambda_1,\dots,\lambda_n$ are eigenvalues of some matrix $X\in \Sym_n$ with diagonal entries $x_1,\dots, x_n\geq 0$, that is, the vector $\lambda=\lambda(X)$ is in $\lambda(\cS^{n,1})$.
\end{proofof}

\bibliographystyle{siam}


\end{document}